\documentclass{amsart}[11pt]
\usepackage{amsmath,amsthm,amsfonts,amssymb,amscd,enumerate,verbatim}
\usepackage{graphicx}
\usepackage[hidelinks=true]{hyperref}

\usepackage[all]{xy}\SelectTips{eu}{}
\usepackage{xcolor}

\usepackage{color}

\newcommand{\xra}{\xrightarrow}

\newcommand{\fm}{{\mathfrak m}}

\newcommand{\n}{{\mathfrak n}}

\newcommand{\R}{{\mathcal R}}

\newcommand{\q}{{\operatorname{q}}}

\newcommand{\cx}{\operatorname{cx}}

\newcommand{\pd}{\operatorname{pd}}

\newcommand{\depth}{\operatorname{depth}}

\newcommand{\Ann}{\operatorname{Ann}}

\newcommand{\HH}{\operatorname{H}}

\newcommand{\Tor}{\operatorname{Tor}}

\newcommand{\Ext}{\operatorname{Ext}}

\newcommand{\Hom}{\operatorname{Hom}}

\newcommand{\supp}{\operatorname{supp}}

\newcommand{\Kos}{\operatorname{K}}

\numberwithin{equation}{section}

\theoremstyle{plain}
\newtheorem{theorem}{Theorem}[section]

\newtheorem*{Main Theorem}{Main Theorem}

\newtheorem{proposition}[theorem]{Proposition}
\newtheorem{lemma}[theorem]{Lemma}
\newtheorem{corollary}[theorem]{Corollary}

{\Alph{theorem}}
{\Alph{theorem}}

\theoremstyle{definition}
\newtheorem{notation}[theorem]{Notation}

\newtheorem{assumptions}[theorem]{Assumptions}
\newtheorem{chunk}[theorem]{}

\newtheorem{remark}[theorem]{Remark}
\newtheorem{definition}[theorem]{Definition}
\newtheorem{example}[theorem]{Example}

\newtheorem{question}[theorem]{Question}

  \newcounter{numlist} %
  {\end{list}}%

\theoremstyle{remark}

\numberwithin{equation}{theorem}

\begin{document}
\title[Dimension and depth inequalities over complete intersections]{Dimension and depth inequalities over complete intersections}

\author{Petter Andreas Bergh}
\address[P. A. Bergh]{Institutt for matematiske fag, NTNU, N-7491 Trondheim, Norway}
\email{petter.bergh@ntnu.no}
\urladdr{https://www.ntnu.edu/employees/petter.bergh}
\author{David A. Jorgensen}
\address[D. A. Jorgensen]{Department of Mathematics, University of Texas at Arlington, 411 S. Nedderman Drive, Pickard Hall 429, Arlington, TX 76019, U.S.A.}
\email{djorgens@uta.edu}
\urladdr{http://www.uta.edu/faculty/djorgens/}
\author{Peder Thompson}
\address[P. Thompson]{Division of Mathematics and Physics, Mälardalen University, Västerås 72123, Sweden}
\email{peder.thompson@mdu.se}
\urladdr{https://sites.google.com/view/pederthompson}

\makeatletter
\@namedef{subjclassname@2020}{%
  \textup{2020} Mathematics Subject Classification}
\makeatother

%\date{{\color{blue}July 29, 2024}}
\thanks{P. A.\ Bergh would like to thank the organizers of the Representation Theory program hosted by the Centre for Advanced Study at The Norwegian Academy of Science and Letters, where he spent parts of fall 2022. P.\ Thompson is grateful for the support provided by the 2022 Summer Research Grant at Niagara University. Part of this work was completed at Centro Internazionale per la Ricerca Matematica in Trento, Italy, in December 2023; all three authors are grateful for this support.}

\begin{abstract}
For a pair of finitely generated modules $M$ and $N$ over a codimension $c$ complete intersection ring $R$ with $\ell(M\otimes_RN)$ finite, we pay special attention to the inequality $\dim M+\dim N \leq \dim R +c$. In particular, we develop an extension of Hochster's theta invariant whose nonvanishing detects equality. In addition, we consider a parallel theory where dimension and codimension are replaced by depth and complexity, respectively.
\end{abstract}

\maketitle

\section*{introduction}
Porting geometric ideas into homological algebra, Serre's development of intersection multiplicity \cite{Serre} provided a foundation for many important homological conjectures. A basic result of Serre's work is that the dimension inequality 
\begin{align*}\tag{1}\label{dimineq1}\dim M+\dim N\leq \dim Q\end{align*}
holds for finitely generated modules $M$ and $N$ over a regular local ring $Q$, as long as the length $\ell(M\otimes_QN)$ is finite. Hochster extended these ideas into the realm of hypersurfaces, defining his theta invariant in \cite{Hoc81} in part as a mechanism for understanding when this dimension inequality holds for singular rings. One aim of the present work is to continue this investigation for complete intersections of arbitrary codimension.

Let $R$ be a complete intersection ring of codimension $c$: that is $R=Q/(f_1,...,f_c)$, where $Q$ is a regular local ring and $f_1,...,f_c$ is a regular sequence. Let $M,N$ be a pair of finitely generated $R$-modules with $\ell(M\otimes_R N)<\infty$. If the pair has complexity $r=\cx^R(M,N)$ 
(see \ref{dfn_cx} for the definition of complexity), we define in Definition \ref{theta_dfn} an invariant $\theta_r^R(M,N)$ as the $(r-1)$-st difference of the numerical function underlying Hochster's theta invariant \cite{Hoc81}. This extends Hochster's theta invariant $\theta^R(M,N)$ to complete intersections of arbitrary codimension. 

One of our first key results, proven in Theorem \ref{redcx}, is the existence of a complete intersection $R'$ over which complexity is reduced and this invariant is preserved. Namely, assuming $r=\cx^R(M,N)\geq 1$, we show that there exists a principal lifting $R'$ of $R$ such that $\cx^{R'}(M,N)=r-1$ and $\theta_{r-1}^{R'}(M,N)=\theta_r^R(M,N)$.  This procedure plays a central role in the remainder of the paper.

As a natural extension of Serre's dimension inequality \eqref{dimineq1} to complete intersections, one has that, for a pair of finitely generated $R$-modules $M,N$ with $\ell(M\otimes_R N)<\infty$, there is a dimension inequality (see Remark \ref{theta_serre}):
\begin{align*}\tag{2}\label{dimineq2}\dim M+\dim N\leq \dim R+c\;.\end{align*}
Hochster referred to $R$ as admissible if $Q$ satisfies the conclusions of Serre's intersection multiplicity conjectures (see \ref{SerreIMC}). If $R$ is an admissible hypersurface (that is, $c=1$), Hochster showed that $\theta^R(M,N)=0$ if and only if the inequality in \eqref{dimineq2} is strict. We show in Theorem \ref{HochsterCI} that if $R$ is an admissible complete intersection of any codimension $c$, vanishing of $\theta_c^R(M,N)$ is equivalent to strict inequality in \eqref{dimineq2}. 

Although it is unknown whether one always has $\dim M+\dim N\leq \dim R+r$, where $r=\cx^R(M,N)$ (indeed, this is related to a long-standing open question about modules of finite projective dimension; see the end of Section \ref{sec_theta}), we show in Theorem \ref{depth_ineq} that 
\begin{align*}\tag{3}\label{dimineq3}\depth M+\depth N\leq \dim R+r\;.\end{align*}
One naturally wonders whether vanishing of $\theta_r^R(M,N)$ might detect whether the inequality \eqref{dimineq3} is strict: an answer is given in Theorem \ref{depth_ineq_strict_vanishing} and Remark \ref{converse}.

Finally, although strictness in the inequality \eqref{dimineq3} is a weaker condition than vanishing of the invariant $\theta_r^R(M,N)$, for sufficiently nice pairs of modules much more can be said. This is the investigation of Section \ref{sec_lift}. Under the assumption that a pair of modules $M,N$ is intersection liftable (this notion is defined in Definition \ref{dfn_int_lift} and includes the assumption that $\ell(M\otimes_RN)$ is finite), Theorem \ref{thm_liftable} shows that $\dim M+\dim N\leq \dim R+r$ holds, $\theta_r^R(M,N)\geq 0$ holds, and $\theta_r^R(M,N)$ vanishes precisely when $\dim M+\dim N<\dim R+r$.

The invariant $\theta_r^R(M,N)$ defined and studied in this paper is related to similar invariants studied elsewhere. For example, Dao defines an invariant in terms of limits in \cite{Dao07} that agrees with our invariant up to a constant, see Remark \ref{rmk_compare}. We found the discussion in \cite{Dao07} regarding the Artinian property of Tor useful; see Section \ref{prelim} below.  Additionally, Moore, Piepmeyer, Spiroff, and Walker \cite{MPSW2} consider a restriction of our invariant to the codimension case for certain complete intersections. In these cases, the primary emphasis was on (non)vanishing of these invariants. The focus here is different: we aim to relate (non)vanishing of such an invariant to inequalities involving dimension and depth---returning more to the ideas considered by Hochster. One could also consider a similar invariant defined in terms of Ext rather than Tor; we decided to focus on an invariant in terms of Tor to relate to the dimension inequalities considered by Serre and Hochster.

The paper is outlined as follows. The first section is devoted to developing preliminary results and recalling definitions about complexity and Hilbert polynomials. In Section \ref{sec_theta} we define $\theta_r^R(M,N)$, and prove in Theorem \ref{redcx} one of our main tools for reducing complexity.  This allows us to prove Theorem \ref{HochsterCI}, motivated by Hochster's result. We also give new proofs of some basic properties of complexity, in particular a new proof of a result of Avramov and Buchweitz regarding complexity (see Theorem \ref{AvrBuc}). Aiming to refine our extension of Hochster's result, we turn in Section \ref{sec_depth} to replacing dimension by depth: the primary result here is Theorem \ref{depth_ineq}, and we obtain a number of applications, including an extension of the Intersection Theorem (see Theorem \ref{int_thm_ext}). Section \ref{sec_highercodim} considers versions of Serre's multiplicity conjectures in the situation of the previous two sections. A natural observation that arises over the course of this investigation is that the main result in Section \ref{sec_theta} should hold at least for sufficiently nice pairs of modules: these are the modules that are studied in Section \ref{sec_lift}. An appendix provides additional details for some proofs.

%%%%%%%%
\section{Preliminaries}\label{prelim}
Let $R$ be a local ring (meaning commutative Noetherian with unique maximal ideal) with residue field $k$.
Let $\R=R[x_1,\dots,x_c]$ be a polynomial ring with indeterminates $x_1,\dots,x_c$, all of the same degree $d$.  Thus $\R$ is a Noetherian graded ring.  We first establish some basic results on Matlis duality for graded $\R$-modules, then set some notation and discuss complexity. Throughout this paper, we use $\ell(-)$ to denote length.

We now want to define a dualizing functor which takes graded $\R$-modules to graded $\R$-modules. The version of graded Matlis duality we give below has been considered in more restrictive situations by Jørgensen \cite{Jor97} and van den Bergh \cite{vdb97}.

\begin{chunk}\label{D} Let $E$ denote the injective hull of $k$ as an $R$-module. 
For any $R$-module $N$ of finite length we have
\begin{enumerate}
\item $\ell(\Hom_R(N,E))=\ell(N)$, in particular, $\Hom_R(N,E)$ has finite length;
\item the natural map $N\to \Hom_R(\Hom_R(N,E),E)$ is an isomorphism.
\end{enumerate} 
(See \cite[Proposition 3.2.12]{BH}.)
For a graded $\R$-module $G=\bigoplus_{n\in\mathbb Z}G_n$ we then define the functor 
\[
D(G)=\bigoplus_{n\in\mathbb Z}\Hom_R(G_n,E)
\] 
which returns a graded $\R$-module. The grading of $D(G)$ is given by 
$D(G)_n=\Hom_R(G_{-n},E)$ for $n\in\mathbb Z$. The action of $x_i$ on $D(G)$ is given by 
$x_ig(a)=g(x_ia)$, for $g\in\Hom_R(G_n,E)=D(G)_{-n}$ and $a\in G_{n-d}$.  Thus 
$x_ig\in\Hom_R(G_{n-d},E)=D(G)_{-n+d}$.  
For a homomorphism $h:G\to G'$ of graded 
$\R$-modules with $h=\bigoplus_{n\in\mathbb Z}h_n$ and each $h_n:G_n\to G_n'$ $R$-linear, we have 
$D(h):D(G')\to D(G)$ given by $D(h)=\bigoplus_{n\in\mathbb Z}\Hom_R(h_n,E)$ and each
$\Hom_R(h_n,E):\Hom_R(G'_n,E)\to\Hom_R(G_n,E)$ is $R$-linear.

Now suppose that the graded $\R$-module $G$ is such that each $G_n$ has finite length as an $R$-module. Then from the discussion above we see that $D(G)$ is again a graded $\R$-module such that each
$D(G)_n$ has finite length as an $R$-module.
\end{chunk}

\begin{theorem} The functor $D$ is exact and $G\cong D(D(G))$ for all graded $\R$-modules $G$ with 
$G_n$ having finite length as an $R$-module for all $n$.
\end{theorem}

\begin{proof} The exactness of the functor $D$ follows directly from the exactness of the functor 
$\Hom_R(-,E)$ from $R$-modules to $R$-modules.

For the second statement, the $R$-linear isomorphisms 
\[
\varphi_n:G_n\to\Hom_R(\Hom_R(G_n,E),E)
\] 
in each degree, mentioned in \ref{D}, are given by $\varphi_n(a)(g)=g(a)$ for $a\in G_n$ and 
$g\in\Hom_R(G_n,E)$. We claim moreover that these constitute an isomorphism of $\R$-modules. Indeed, for $a\in G_{n-d}$ and $g\in\Hom_R(G_n,E)$ we have $\varphi_n(x_ia)(g)=g(x_ia)=x_ig(a)=\varphi_{n-d}(a)(x_ig)=x_i\varphi_{n-d}(a)(g)$. Thus
$\varphi_n(x_ia)=x_i\varphi_{n-d}(a)$, and so $\varphi=\bigoplus_{n\in\mathbb Z}\varphi_n$ is 
$\R$-linear.
\end{proof}

Recall that a graded $\R$-module $G=\bigoplus_{n\in \mathbb Z}G_n$ is \emph{Artinian} if it satisfies DCC on graded submodules, and \emph{Noetherian} if it satisfies ACC on graded submodules. We have the following slight variation of results by Kirby \cite{K}, which give equivalent conditions for Artinian and Noetherian graded $\R$-modules $G$.

\begin{theorem}\label{Kirby} Let $G=\bigoplus_{n\in \mathbb Z}G_n$ be a graded $\R$-module. Assume that the $x_i$ all have negative degree $d<0$. Then
\begin{enumerate}
\item $G$ is a Noetherian $\R$-module if and only if there exist integers $r$ and $s$ such that 
\begin{enumerate}[(i)]
\item $G_n=0$ for $n>r$;
\item $G_{n+d}=x_1G_n+\cdots +x_cG_n$ for all $n\le s$;
\item $G_n$ is a Noetherian $R$-module for all $n$.
\end{enumerate}
\item $G$ is an Artinian $\R$-module if and only if there exist integers $r$ and $s$ such that 
\begin{enumerate}[(i)]
\item $G_n=0$ for $n<r$;
\item $(0:_{G_n}Rx_1+\cdots+Rx_c)=0$ for all $n\ge s$;
\item $G_n$ is an Artinian $R$-module for all $n$.
\end{enumerate}
\end{enumerate}
\end{theorem}

We provide a proof in the appendix.  Note that Condition (1)(ii) of Theorem \ref{Kirby} is equivalent to the $R$-linear maps $\rho^G_n:G_n^c\to G_{n+d}$ given by 
$\rho^G_n(a_1,\dots,a_c)=x_1a_1+\cdots +x_ca_c$ being surjective for $n\le s$.  Similarly, 
Condition (2)(ii) of Theorem \ref{Kirby} is equivalent to the $R$-linear maps
$\iota^G_n:G_n\to G^c_{n+d}$ given by $\iota^G_n(a)=(x_1a,\dots,x_ca)$ being injective for $n\ge s$.

One checks easily that $D(\rho^G_n)=\iota^{D(G)}_{-n-d}$ and $D(\iota^G_n)=\rho^{D(G)}_{-n-d}$.

\begin{theorem}\label{ArtNoeth}
Let $G$ be a graded $\R$-module such that each $G_i$ has finite length as an $R$-module. Assume that the $x_i$ all have negative degree $d<0$. Then $G$ is Noetherian (resp. Artinian) if and only if $D(G)$ is Artinian (resp. Noetherian). 
\end{theorem}

\begin{proof} Suppose that $G$ is Noetherian.  Then, according to Theorem 
\ref{Kirby}, there exist integers $r$ and $s$ such that 
\begin{enumerate}[(i)]
\item $G_n=0$ for $n>r$;
\item $\rho^G_n$ is surjective for all $n\le s$;
\item $G_n$ is a Noetherian $R$-module for all $n$.
\end{enumerate}
Thus applying $D$ we see that
\begin{enumerate}[(i)]
\item $D(G)_n=D(G_{-n})=0$ for $n<-r$;
\item $\iota^{D(G)}_n=D(\rho^G_{-n-d})$ is injective for all $n\ge -s+d$;
\item $D(G)_n=D(G_{-n})$ is an Artinian $R$-module for all $n$.
\end{enumerate}
Thus by Theorem \ref{Kirby} again, we see $D(G)$ is Artinian.  The converse is similar.
\end{proof}

\begin{notation}
\label{Rnotation} 
Throughout this paper, let $R=Q/(f_1,\dots,f_c)$, where $Q$ is a local ring and 
$f_1,\dots,f_c$ is a $Q$-regular sequence contained in the square of the maximal ideal of $Q$. We denote by $k$ the common residue field of $R$ and $Q$. If we say that $R$ is a complete intersection, we mean that $Q$ is a regular local ring. We also assume $M$ and $N$ are finitely generated $R$-modules.
We let $\Tor^R(M,N)$ denote $\bigoplus_{n\ge 0}\Tor_n^R(M,N)$.
\end{notation}

\begin{assumptions}\label{Gu} We will assume throughout the paper that $\ell(M\otimes_RN)$, the length of $M\otimes_RN$, is finite. In this case, each $\Tor_n^R(M,N)$ and $\Tor_n^Q(M,N)$ has finite length, $n\ge 0$. 

From now on we assume that $\R=R[x_1,\dots,x_c]$ is the polynomial ring in $c$ variables, each of degree $d=-2$. Gulliksen shows in \cite{Gu} that $T=\Tor^R(M,N)$ can be given the structure of a graded module over $\R$ such that $T$ is an Artinian $\R$-module whenever $\Tor^Q(M,N)$ is an Artinian 
$Q$-module. We assume throughout that this is the case.  We note that since each $\Tor^Q_n(M,N)$ has finite length, the condition that $\Tor^Q(M,N)$ is an Artinian $Q$-module is equivalent to 
the condition that $\Tor^Q_n(M,N)=0$ for all $n\gg 0$. This holds, for example, if $R$ is a complete intersection. 
\end{assumptions}

Since each $x_i\in\R$ acts on $T$ with degree $-2$, the graded $\R$-module $T$ naturally decomposes as a direct sum of graded $\R$-modules $T_e$, the Tors with even index, and $T_o$, the Tors with odd index. Obviously $T$ is Artinian if and only if both $T_e$ and $T_o$ are Artinian. 

\begin{chunk}
Let $H_e^R(M,N)$ and $H_o^R(M,N)$ denote the even and odd Hilbert functions of the pair $M,N$.  That is,
\[
H_e^R(M,N)(n)=\ell(\Tor_{2n}^R(M,N)) \qquad \text{for } n\ge 0
\]
and 
\[
H_o^R(M,N)(n)=\ell(\Tor_{2n-1}^R(M,N)) \qquad \text{for } n\ge 1
\]
Since $T_e$ and $T_o$ are Artinian graded $\R$-modules, $D(T_e)$ and $D(T_o)$ are Noetherian graded 
$\R$-modules, and so their Hilbert functions are of polynomial type of degree at most $c-1$; this follows, for example, from a slight modification of \cite[Theorem 4.1.3]{BH}. Graded Matlis duality \ref{D} also gives the equalities
\[
\ell(\Tor_n^R(M,N))=\ell(D(\Tor_n^R(M,N))) \qquad \text{for } n\ge 0
\]
It follows that the functions $H_e^R(M,N)$ and $H_o^R(M,N)$ are also of polynomial type of degree at most $c-1$. Thus there exist polynomials 
$P_e^{R,M,N}(x), P_o^{R,M,N}(x)\in\mathbb Q[x]$ such that 
\[
P_e^{R,M,N}(n)=H_e^R(M,N)(n)
\] 
and 
\[
P_o^{R,M,N}(n)=H_o^R(M,N)(n)
\] 
for all $n\gg 0$. We call $P_e^{R,M,N}(x)$ and $P_o^{R,M,N}(x)$ the \emph{even and odd Hilbert polynomials} of the pair $M,N$. 
\end{chunk}

\begin{definition}\label{dfn_cx} Assuming $\ell(M\otimes_RN)<\infty$ and 
$\Tor^Q_n(M,N)=0$ for all $n\gg 0$, we define the \emph{complexity}, denoted $\cx^R(M,N)$, of the pair $M,N$ over $R$ as
\[
\cx^R(M,N)=\max\{\deg P_e^{R,M,N}(x),\deg P_o^{R,M,N}(x)\}+1
\]
where by convention we take the degree of the zero polynomial to be $-1$.  In this case we have 
$\cx^R(M,N)=0$ if and only if $P_e^{R,M,N}(x)=P_o^{R,M,N}(x)=0$. We write $\cx^RM$ to mean
$\cx^R(M,k)$. Contrary to some extant literature, we write $\cx^R(M,N)$ instead of $\cx_R(M,N)$ to reflect that it is Tor complexity, rather than Ext complexity.
\end{definition}

\begin{chunk}\label{ab} Thus for $r=\cx^R(M,N)$ we may write
\[
P_e^{R,M,N}(x)=ax^{r-1}+g(x)
\] 
and 
\[
P_o^{R,M,N}(x)=bx^{r-1}+h(x)
\]
where $a$ and $b$ are rational numbers, and $\deg g(x)<r-1$ and $\deg h(x)<r-1$.
\end{chunk}

\begin{remark} Our definition of complexity of the pair $M,N$ is not the ordinary Ext complexity of the pair $M,N$ defined in \cite{AB}.  It is however the same as the length complexity as defined in 
\cite{Dao07}.  Both complexities agree over a complete intersection, when defined, 
according to \cite[Theorem 5.4]{Dao07}.
\end{remark}

%%%%%%%%
\section{An extension of Hochster's theta invariant}\label{sec_theta}
In this section we introduce a new invariant $\theta_s^R(M,N)$ that extends Hochster's theta invariant $\theta^R(M,N)$ for hypersurfaces to complete intersections of higher codimension. We then turn to proving one of our central results, Theorem \ref{redcx}, which provides the existence of a principal lifting of $R$ that reduces complexity of the pair of modules and preserves this new invariant. With this in place, we establish an extension of Hochster's result, prove results on complexity and properties of this new invariant, and consider a well-known open question.

Let $R=Q/(f_1,...,f_c)$ be as in Notation \ref{Rnotation}, where $Q$ is a local ring and $f_1,...,f_c$ is a $Q$-regular sequence contained in the square of the maximal ideal of $Q$.

Consider an arbitrary numerical function $F:\mathbb Z\to\mathbb Z$. The (first backward) difference operator $\nabla$ applied to $F$ is defined to be
\[
\nabla F(n) = F(n)-F(n-1)
\]
One extends this operator to higher order $r\ge 1$ by setting 
$\nabla^rF=\nabla(\nabla^{r-1}F)$, where $\nabla^1=\nabla$, and $\nabla^0F=F$. A straightforward computation shows that 
\begin{equation}
\nabla^rF(n)=\sum_{i=0}^r(-1)^i\binom r i F(n-i)
\end{equation}
Note that $\nabla$ takes a nonzero polynomial function to a polynomial function of degree one less (we assume throughout that $\deg 0=-1$).

We extend Hochster's theta invariant as follows. Consider the numerical function
\[
\theta^R(M,N)=H_e^R(M,N)-H_o^R(M,N),
\]
that is, 
\begin{equation}\label{theta_function}
\theta^R(M,N)(n)=\ell(\Tor_{2n}^R(M,N))-\ell(\Tor_{2n-1}^R(M,N)) 
\qquad \text{for } n\ge 1 .
\end{equation}
For $r\ge 0$ consider the numerical function $\nabla^r\theta^R(M,N)$, whose explicit expression is
\begin{equation}\label{nabr}
\nabla^r\theta^R(M,N)(n)=
\sum_{i=0}^r(-1)^i\binom r i \left(\ell(\Tor_{2(n-i)}^R(M,N))-\ell(\Tor_{2(n-i)-1}^R(M,N))\right)
\end{equation}

Recall from \ref{ab} that $a$ and $b$ are the degree $r-1$ coefficients of the even and odd Hilbert polynomials of the pair $M,N$.

\begin{proposition}\label{prop1}  Let $M$ and $N$ be finitely generated $R$-modules with $\ell(M\otimes_R N)$ finite and $\Tor^Q_n(M,N)=0$ for all $n\gg 0$. Suppose that $r=\cx^R(M,N)$. Then we have:
\begin{enumerate}
\item if $s\ge r$, then $\nabla^s\theta^R(M,N)(n)=0$ for all $n\gg 0$;
\item if $r\geq 1$, then $\nabla^{r-1}\theta^R(M,N)(n)=(r-1)!(a-b)$ for all $n\gg 0$. Thus $\nabla^{r-1}\theta^R(M,N)$ is eventually constant, and 
$\nabla^{r-1}\theta^R(M,N)(n)=0$ for all $n\gg 0$ if and only if 
$P_e^{R,M,N}(x)$ and $P_o^{R,M,N}(x)$ have the same degree and leading coefficient.
\end{enumerate}
\end{proposition}

\begin{proof}
(1) is immediate from the definitions. (2) follows by induction on $r$, using that for $t\geq 0$, the difference operator $\nabla^t$ is $\mathbb{Q}$-linear and satsifies $\nabla^t(n^t)=t!$.
\end{proof}

\begin{definition}\label{theta_dfn} Let $M$ and $N$ be finitely generated $R$-modules with $\ell(M\otimes_R N)$ finite. Suppose that $r=\cx^R(M,N)\ge 1$. Then for $s\ge r$ we set
\[
\theta_s^R(M,N) = \nabla^{s-1}\theta^R(M,N)(n) 
\]
for $n\gg 0$. Note that by Proposition \ref{prop1}, 
$\theta_s^R(M,N)=(r-1)!(a-b)$ if $s=r$ and $\theta_s^R(M,N)=0$ if $s>r$.
\end{definition}

\subsection*{Principal liftings}
In this subsection, we establish a key tool used throughout the paper, Theorem \ref{redcx}, which says that we can change rings to reduce the complexity of a pair of modules while simultaneously preserving this new invariant. 

\begin{definition} We say that $R'$ is a \emph{principal lifting} of $R$ if there exists a set of generators $g_1,\dots,g_c$ of $(f_1,\dots,f_c)$ such that $R'=Q/(g_1,\dots,g_{c-1})$.  Note in this case we have $R=R'/(g_c)$, and so $\dim R'=\dim R+1$.
\end{definition}

The existence of short exact sequences of Tor as in the following lemma are guaranteed by Theorem \ref{redcx} below.

\begin{lemma}\label{theta} Let $M$ and $N$ be finitely generated $R$-modules with $\ell(M\otimes_R N)$ finite.
Let $R'$ be a principal lifting of $R$, and assume that for all large $n$ we have
short exact sequences
\[
0\to \Tor_n^{R'}(M,N) \to \Tor_n^R(M,N) \to \Tor_{n-2}^R(M,N) \to 0
\]
Then for $r\ge 1$ we have 
\[
\nabla^{r-1}\theta^{R'}(M,N)(n)=\nabla^r\theta^R(M,N)(n)
\]
for all $n\gg 0$.
\end{lemma}

\begin{proof} Abbreviate $\ell(\Tor^R_i(M,N))$ by $\ell_i^R$.  Then, using \eqref{nabr}, 
for all $n\gg 0$ we have
\begin{align*}
\nabla^{r-1}&\theta^{R'}(M,N)(n)=\sum_{i=0}^{r-1}(-1)^i\tbinom{r-1}i 
\left(\ell_{2(n-i)}^{R'}-\ell_{2(n-i)-1}^{R'}\right)\\
=&\sum_{i=0}^{r-1}(-1)^i\tbinom{r-1}i\ell_{2(n-i)}^{R'}-\sum_{i=0}^{r-1}(-1)^i\tbinom{r-1}i\ell_{2(n-i)-1}^{R'}\\
=&\sum_{i=0}^{r-1}(-1)^i\tbinom{r-1}i\left(\ell_{2(n-i)}^{R}-\ell_{2(n-i)-2}^R\right)-\sum_{i=0}^{r-1}(-1)^i\tbinom{r-1} i\left(\ell_{2(n-i)-1}^R-\ell_{2(n-i)-3}^R\right)\\
\intertext{Now we exchange the second term in parentheses in the first sum with the first term in parentheses in the second sum, getting}
=&\sum_{i=0}^{r-1}(-1)^i\tbinom{r-1} i\left(\ell_{2(n-i)}^{R}-\ell_{2(n-i)-1}^R\right)-\sum_{i=0}^{r-1}(-1)^i\tbinom{r-1} i\left(\ell_{2(n-i)-2}^R-\ell_{2(n-i)-3}^R\right)\\
=&\sum_{i=0}^{r-1}(-1)^i\tbinom{r-1} i\left(\ell_{2(n-i)}^{R}-\ell_{2(n-i)-1}^R\right)-\sum_{i=1}^r(-1)^{i-1}\tbinom{r-1}{i-1}\left(\ell_{2(n-i)}^R-\ell_{2(n-i)-1}^R\right)\\
=&\sum_{i=0}^r(-1)^i\left(\tbinom{r-1}i+\tbinom{r-1}{i-1}\right)
\left(\ell_{2(n-i)}^R-\ell_{2(n-i)-1}^R\right)\\
=&\sum_{i=0}^r(-1)^i\tbinom ri\left(\ell_{2(n-i)}^R-\ell_{2(n-i)-1}^R\right)\\
=&\nabla^r\theta^R(M,N)(n)
\end{align*}
\end{proof}

\begin{chunk}\label{injsurj} Let $G$ be a graded $\R$-module. It follows immediately from the exactness of $D$ that multiplication by a homogeneous element $x\in\R$ is eventually surjective on $G$, meaning $G_n\xra{x} G_{n+|x|}\to 0$ is exact for all $n\gg 0$, if and only if multiplication by $x$ on $D(G)$ is eventually injective: $0\to D(G)_n \xra{x} D(G)_{n+|x|}$ is exact for all $n\ll 0$. 
\end{chunk}

The following is a key result that allows for induction by passing to a principal lifting; it is a slight generalization of \cite[Theorem 3.1]{Eis80}.

\begin{lemma}\label{coreg} 
Assume that the residue field $k$ is infinite. Let $M$ and $N$ be finitely generated $R$-modules with $\ell(M\otimes_R N)$ finite and $\Tor^Q_n(M,N)=0$ for all $n\gg 0$. Consider the graded $\R$-module $T=\bigoplus_{n\ge 0}T_n$, where 
$T_n=\Tor^R_n(M,N)$. Then there exists a linear form $x\in \R_{-2}$ such that $T_{n-2}=xT_n$ 
for all $n\gg 0$.
\end{lemma}

\begin{proof} By \ref{Gu} we have that $T$ is an Artinian $\R$-module. Therefore by Theorem 
\ref{ArtNoeth}, $D(T)$ is a Noetherian $\R$-module. Now by Theorem \ref{eres} in the appendix, there exists $x\in \R_{-2}$ such that $x$ is eventually regular on $D(T)$, that is, multiplication by $x$ is eventually injective.  Thus $x$ is eventually surjective on $T$, by \ref{injsurj}.
\end{proof}

\begin{theorem}\label{redcx} 
Assume that the residue field $k$ is infinite. Let $M_1,\dots,M_l$ and $N_1,\dots,N_m$ be finitely generated $R$-modules such that $\ell(M_i\otimes_RN_j)$ is finite and 
$\Tor^Q_n(M_i,N_j)=0$ for all $n\gg 0$ and all $i,j$.  Then there exists a principal lifting $R'$ of $R$ such that the following hold for all $i,j$:
\begin{enumerate}
\item For all $n\gg0$, there are exact sequences:
\[
0\to \Tor_n^{R'}(M_i,N_j) \to \Tor_n^R(M_i,N_j) \to \Tor_{n-2}^R(M_i,N_j) \to 0\;.
\]
\item If $\cx^R(M_i,N_j)\ge 1$, then 
\[
\cx^{R'}(M_i,N_j)=\cx^R(M_i,N_j)-1\;.
\] 
\item For $s\ge\cx^R(M_i,N_j)$,
\[
\theta_{s-1}^{R'}(M_i,N_j)=\theta_s^R(M_i,N_j)\;.
\]
\end{enumerate}
\end{theorem}

\begin{proof}  
By Lemma \ref{coreg} there exists a homogeneous element $x\in\mathcal R$ of 
degree $-2$ which is eventually surjective on 
$\Tor^R(\oplus_{i=1}^lM_i,\oplus_{j=1}^mN_j)$.  This implies that $x$ is eventually surjective on each summand $\Tor^R(M_i,N_j)$.
By \ref{pls} and \ref{eisops} in the appendix, there exists a principal lifting $R'$ of $R$ and change of rings long exact sequences:
\[
\cdots \to \Tor^{R'}_n(M_i,N_j)\to\Tor^R_n(M_i,N_j)\xrightarrow{x}\Tor^R_{n-2}(M_i,N_j) \to \cdots
\]
which eventually break up into short exact sequences for all $i,j$, since multiplication by $x$ is eventually surjective. This proves (1). 

Now set $M=M_i$ and $N=N_j$.  The short exact sequences from (1) yield the equations
\[
H_e^{R'}(M,N)(n)=H_e^R(M,N)(n)-H_e^R(M,N)(n-1)=\nabla H_e^R(M,N)(n)
\]
and
\[
H_o^{R'}(M,N)(n)=H_o^R(M,N)(n)-H_o^R(M,N)(n-1)=\nabla H_o^R(M,N)(n)
\]
for all $n\gg 0$. Thus $\deg P_e^{R',M,N}(x)=\deg P_e^{R,M,N}(x)-1$ and
$\deg P_o^{R',M,N}(x)=\deg P_o^{R,M,N}(x)-1$.  This establishes (2), and (3) follows from Lemma \ref{theta}.
\end{proof}

\begin{remark}\label{infk}
One may always find a faithfully flat ring extension $Q\to \widetilde{Q}$ where $\widetilde{Q}$ has an infinite residue field, see for example the proof of \cite[Proposition 2.3]{J2}. Taking $\widetilde{R}=\widetilde{Q}/(\widetilde{f_1},...,\widetilde{f_c})$, where $\widetilde{f_i}$ is the image of $f_i$ in $\widetilde{Q}$, we obtain a faithfully flat ring extension $R\to \widetilde{R}$. In this process, $\widetilde{Q}$ is again a regular local ring, $\widetilde{f_1},...,\widetilde{f_c}$ is a regular sequence in $\widetilde{Q}$, and hence $\widetilde{R}$ is a complete intersection. If $M$ and $N$ are finitely generated $R$-modules, then the following hold: 
\begin{itemize}
\item $\Tor_n^{\widetilde{R}}(M\otimes_R \widetilde{R},N\otimes_R \widetilde{R})\cong \Tor_n^R(M,N)\otimes_R \widetilde{R}$
\item $\dim_R M=\dim_{\widetilde{R}}M\otimes_R \widetilde{R}$
\item $\ell_R(M)=\ell_{\widetilde{R}}(M\otimes_R \widetilde{R})$
\item $\depth_R M=\depth_{\widetilde{R}}M\otimes_R \widetilde{R}$
\item $\cx_R(M,N)=\cx_{\widetilde{R}}(M\otimes_R \widetilde{R},N\otimes_R\widetilde{R})$
\item $\theta_s^R(M,N)=\theta_s^{\widetilde{R}}(M\otimes_R\widetilde{R},N\otimes_R\widetilde{R})$.
\end{itemize}
\end{remark}

\subsection*{An extension of Hochster's result}
In this subsection, we now extend Hochster's result from hypersurfaces to complete intersections of arbitrary codimension. 

Recall that Serre's Euler characteristic $\chi^R(M,N)$ for intersection multiplicities in the case where $\ell (M\otimes_R N)<\infty$ and $\cx^R(M,N)=0$ is given by
\[
\chi^R(M,N)=\sum_{i\ge 0}(-1)^i\ell(\Tor_i^R(M,N))\;.
\]

\begin{lemma}\label{chi=theta} 
Let $M$ and $N$ be finitely generated $R$-modules with $\ell(M\otimes_R N)$ finite and 
$\Tor^Q_n(M,N)=0$ for all $n\gg 0$. 
Suppose that $\cx^R(M,N)\le 1$. Then for any principal lifting $R'$ of $R$ for which $\cx^{R'}(M,N)=0$ we have
\[
\chi^{R'}(M,N)=\theta^R_1(M,N)\;.
\]
In particular, if $\cx^R(M,N)=0$, then $\chi^{R'}(M,N)=0$.
\end{lemma}

\begin{proof} For short, write $T^R_n$ for $\Tor_n^R(M,N)$.  We have 
$R=R'/(f)$ and the change of rings long exact sequence of Tor \eqref{corles}
\begin{align*}
\cdots\to T^R_{p+1} \to T^R_{p-1} \to T^{R'}_p \to \cdots&\\ 
\to T^R_3 \to& T^R_1\to T^{R'}_2 \to T^R_2 \to T^R_0 \to T^{R'}_1 \to T^R_1 \to 0
\end{align*}
Since $\cx^{R'}(M,N)=0$, we have $T^{R'}_n=0$ for all $n\gg 0$. Let 
$q=\sup\{i\mid T^{R'}_i\ne 0\}$. It follows from the long exact 
sequence that $T^R_{q+2n}=T^R_q$ and $T^R_{q+1+2n}=T^R_{q+1}$
for all $n\ge 0$. If we truncate the long exact sequence at 
$T^{R'}_{q+1}=0$, then the alternating sum of lengths of terms equates to 0:
\begin{align*}
(-1)^q\ell T_{q+1}^R+\sum_{i=1}^q(-1)^{i-1}(\ell T^R_{i-1}-\ell T^{R'}_i+\ell T^R_i)=0
\end{align*}
Since the sum is mostly telescoping, it is reduced to 
\begin{equation}\label{q}
\left(\sum_{i=0}^q(-1)^i\ell T^{R'}_i\right) - 
(-1)^q(\ell T^R_q-\ell T^R_{q+1})=0
\end{equation}
Note that if $\cx^R(M,N)=0$, then $\sup\{i\mid T^R_i\ne 0\}=q-1$. If $q$ is even then $(-1)^q(\ell T^R_q-\ell T^R_{q+1})=
\ell T^R_{q+2}-\ell T^R_{q+1} = \theta_1^R(M,N)$.  If $q$ is odd 
then $(-1)^q(\ell T^R_q-\ell T^R_{q+1})=
\ell T^R_{q+1}-\ell T^R_q = \theta_1^R(M,N)$. Thus \eqref{q} says
\[
\chi^{R'}(M,N)-\theta^R_1(M,N)=0
\]
as desired. 

The last statement is obvious from the long exact sequence of Tor above.
\end{proof}

Lemma \ref{chi=theta} validates defining $\theta_0^R(M,N)=\chi^R(M,N)$ when 
$\cx^R(M,N)=0$.  Thus we make this so.

\begin{chunk}\label{SerreIMC}
Recall Serre's intersection multiplicity conjectures for finitely generated modules $M$ and $N$ over a regular local ring $Q$, satisfying $\ell(M\otimes_QN)<\infty$:
\begin{enumerate}
\item Dimension inequality: $\dim M+\dim N\le\dim Q$.
\item Non-negativity: $\chi^Q(M,N)\geq 0$.
\item Vanishing: If $\dim M+\dim N<\dim Q$, then $\chi^Q(M,N)=0$.
\item Positivity: If $\dim M+\dim N=\dim Q$, then $\chi^Q(M,N)>0$.
\end{enumerate}
Serre proved (1)-(4) hold in the case where the completion $\widehat{R}$ is a formal power series over a field or DVR, and showed that (1) holds in general \cite{Serre}. Gabber (see \cite{Rob98Gab}) settled (2), and (3) was shown by Roberts \cite{Rob85} and Gillet-Soul\'{e} \cite{GS87} independently. The positivity conjecture (4) is still open in general.
\end{chunk}

One of the main results of this section is the following extension of Hochster's Theorem 
\cite[Theorem 1.4]{Hoc81}. By saying that $R$ is an admissible complete intersection we mean that $R$ is a complete intersection and that $Q$ satisfies the conclusions of Serre's intersection multiplicity conjectures 
\ref{SerreIMC}.

\begin{theorem}\label{HochsterCI} Let $R$ be an admissible complete intersection of codimension $c$, and let $M$ and $N$ be finitely generated $R$-modules with $\ell(M\otimes_R N)<\infty$. The following are then equivalent: 
 \begin{enumerate}[\quad\rm(1)]
 \item $\theta_c^R(M,N)=0$;
 \item $\dim M+\dim N<\dim R+c$. 
 \end{enumerate}
\end{theorem}

\begin{proof} 
We induct on $c$.  For $c=1$, the statement is just Hochster's result on admissible hypersurfaces. We include the proof for completeness. We have $R=Q/(f)$, and because $\cx^Q(M,N)=0$, we can apply Lemma 
\ref{chi=theta} to conclude that $\chi^Q(M,N)$ and $\theta_1^R(M,N)$ vanish simultaneously. By Serre's theorem, the vanishing of $\chi^Q(M,N)$ is equivalent to 
$\dim M +\dim N<\dim Q=\dim R+1$, and we are done in this case.

Now assume that $c>1$. By Remark \ref{infk}, we may assume that $R$ has an infinite residue field. Suppose that $\cx^R(M,N)\ge 1$.  Then by Theorem \ref{redcx} there exists a principal lifting $R'$ of $R$ such that $\cx^{R'}(M,N)=\cx^R(M,N)-1$, and so $\theta^{R'}_{c-1}(M,N)=\theta^R_c(M,N)$, regardless of whether $\cx^R(M,N)=c$ or $\cx^R(M,N)<c$. Now if $\cx^R(M,N)=0$, then $\cx^{R'}(M,N)=0$ for any principal lifting $R'$ of $R$.  Thus by definition, 
$\theta_{c-1}^{R'}(M,N)=\theta^R_c(M,N)=0$. Thus in any case, the vanishing of $\theta^R_c(M,N)$
is equivalent to the vanishing of $\theta^{R'}_{c-1}(M,N)$, and by induction this is equivalent to 
$\dim M +\dim N<\dim R' +c-1=\dim R +c$.
\end{proof}

\begin{remark}\label{theta_serre}
Since $\dim Q=\dim R+c$, a natural higher codimension analogue of Serre's dimension inequality (1) above holds: for finitely generated $R$-modules $M$ and $N$ satisfying $\ell(M\otimes_R N)<\infty$, one has 
\begin{equation}\label{eqn_dim}\dim M + \dim N \leq \dim R+ c\;.\end{equation}
This theorem therefore provides higher codimension analogues of Serre's conjectures (3) and (4). Note, however, that the naive analogue of (2) fails: $\theta_r^R(M,N)$ can be negative, see Section \ref{sec_highercodim}.
\end{remark}

\begin{remark}\label{rmk_compare}
In the case that $R$ is a hypersurface, i.e., a complete intersection with $c=1$, our invariant 
$\theta_1^R(M,N)$ in Definition \ref{theta_dfn} is precisely Hochster's theta invariant $\theta^R(M,N)$ from \cite{Hoc81}, which has been studied in depth by others. 

For a complete intersection $R$ more generally, and a pair of $R$-modules $M,N$ with $r=\cx^R(M,N)$ and 
$s\geq r$ an integer, our invariant $\theta_s^R(M,N)$ and Dao's invariant $\eta_s^R(M,N)$ from 
\cite{Dao07} vanish simultaneously. In fact, they satisfy 
$$(2^s\cdot s!)\eta_s(M,N)=\theta_s^R(M,N)$$ 
for $s\ge r\geq 0$: Indeed, when $s>r$, both invariants vanish. When $s=r$, this equality follows from induction on $r$: for $r=0$, one has $\theta_0^R(M,N)=\chi^R(M,N)=\eta_0^R(M,N)$, and for $r\geq 1$, without loss of generality we may assume that $R$ has an infinite residue field (see Remark \ref{infk}), so there is a principal lifting $R'$ of $R$ by Theorem \ref{redcx} so that induction, along with the fact that $\eta_{r-1}^{R'}(M,N)=2r\eta_r^R(M,N)$ per \cite[Theorem 4.2]{Dao07}, yields
$$\theta_r^R(M,N)=\theta_{r-1}^{R'}(M,N)=2^{r-1}(r-1)!\eta_{r-1}^{R'}(M,N)=2^rr!\eta_r^R(M,N)\;.$$

Dao's invariant is defined in terms of limits. The invariant $\eta_c(M,N)$ defined (only for the codimension $c$) in \cite{MPSW2} agrees with Dao's invariant, yet, like our invariant, is defined in terms of the $(c-1)$-st difference of Hochster's theta invariant.
\end{remark}

\subsection*{Complexity}

We state some properties of complexity.  Analogous results for ordinary Ext complexity are known.  Our proofs for length Tor complexity are totally different.

The next two results are straightforward to show for Ext complexity using the theory of support varieties, but the proofs here we give for Tor length complexity are different.
For example, the following inequality can also be deduced from \cite[Corollary 5.7]{AB} and \cite[Theorem 5.4]{Dao07}, but we include a proof here since it is short and fundamentally different.

\begin{theorem}\label{AvrBuc} 
Assume that $M$ and $N$ are finitely generated $R$-modules with 
$\ell(M\otimes_RN)<\infty$ and $\Tor^Q_n(M,N)=\Tor^Q_n(M,k)=\Tor^Q_n(N,k)=0$ for all $n\gg 0$. Then the inequality
\[
\cx^R(M,N)\ge \cx^RM+\cx^RN-c
\]
holds.
\end{theorem}

\begin{proof} We induct on $r=\cx^R(M,N)$.  Suppose $r=0$.  Since $\cx^R(M,N)=\cx^R(\Omega_iM,N)$ and $\cx^RM=\cx^R\Omega_iM$ for any $R$-syzygy $\Omega_iM$ of $M$ (note we also have 
$\Tor^Q_n(\Omega_iM,N\oplus k)=0$ for all $n\gg 0$), we can assume that $\Tor^R_n(M,N)=0$ for all $n>0$. In this case, if $F$ and $G$ are minimal free resolutions of $M$ and $N$, respectively, then $F\otimes_RG$ is a minimal free resolution of $M\otimes_R N$.  It follows that 
\[
P_e^{R,M\otimes N,k}(n)=\sum_{i\ge 0}P^{R,M,k}_e(i)P^{R,N,k}_e(n-i)
+\sum_{i\ge 1}P_o^{R,M,k}(i)P_o^{R,N,k}(n-i)
\] 
and 
\[
P_o^{R,M\otimes N,k}(n)=\sum_{i\ge 0}P^{R,M,k}_e(i)P^{R,N,k}_o(n-i)
+\sum_{i\ge 1}P_o^{R,M,k}(i)P_e^{R,N,k}(n-i)
\]  
Now using the well-known fact that if $F(n)$ is a numerical function of polynomial type of degree $d$ and $G(n)$ is a numerical function of polynomial type of degree $e$, then $(F*G)(n)=\sum_{i\ge 0}F(i)G(n-i)$ is a numerical function of polynomial type in $n$ of degree $d+e+1$, we see that $\cx^R(M\otimes_RN)=\cx^RM+\cx^RN$. Since $\cx^R(M\otimes_R N)\leq c$, the desired inequality holds in this case.

For the inductive step, assume now that $r>0$. Then, necessarily, $\cx^RM>0$ and $\cx^RN>0$.  By Remark \ref{infk}, we may assume that $R$ has an infinite residue field. We now employ Theorem \ref{redcx} to find a principal lifting $R'$ of $R$ such that 
$\cx^{R'}(M,N)=\cx^R(M,N)-1$, $\cx^{R'}M=\cx^RM-1$, and $\cx^{R'}N=\cx^RN-1$.  Thus by induction,
\begin{align*}
\cx^R(M,N)&=\cx^{R'}(M,N)+1\\
&\ge\cx^{R'}M+\cx^{R'}N-(c-1)+1\\
&=\cx^RM+\cx^RN-c 
\end{align*}
as claimed.
\end{proof}

\begin{theorem}\label{cxlec} 
Let $M$ and $N$ be finitely generated $R$-modules with $\ell(M\otimes_R N)$ finite and
$\Tor^Q_n(M,N\oplus k)=0$ for all $n\gg 0$. Then the following hold.
\begin{enumerate}
\item $\cx^R(M,N)\le \cx^RM$.
\item If $0\to M_1\to M_2\to M_3\to 0$ and $0\to N_1\to N_2\to N_3\to 0$ are short
exact sequences of finitely generated $R$-modules with $\ell(M_i\otimes_RN)$,
$\ell(M\otimes_RN_j)$ finite and $\Tor^Q_n(M\oplus M_i,N\oplus N_j)=0$ for all $i,j$ and $n\gg 0$, then for $\{h,i,j\}=\{1,2,3\}$ we have
\[
\cx^R(M_h,N)\leq\max\{\cx^R(M_i,N),\cx^R(M_j,N)\}
\]
and
\[
\cx^R(M,N_h)\leq\max\{\cx^R(M,N_i),\cx^R(M,N_j)\}
\]
\item If $\ell(M)$ is finite and $\Tor^Q_n(M,M\oplus k)=0$ for all $n\gg 0$, then
$\cx^R(M,M)=\cx^RM$.
\item If $R$ is a complete intersection and $M$ is a Cohen-Macaulay module of grade $g$, then 
$\cx^RM=\cx^R\Ext^g_R(M,R)$.
\end{enumerate}
\end{theorem}

\begin{proof} 
By Remark \ref{infk}, we may assume that $R$ has an infinite residue field.

(1). We first note that if $\cx^RM=0$, then $\cx^R(M,N)=0$, since the former condition is equivalent to $M$ having finite projective dimension.  Now induct on $r=\cx^R(M,N)$.  When $r=0$ the inequality certainly holds.  Now assume that $r>0$.  Then $\cx^RM>0$, and by Theorem \ref{redcx} one may find a principal lifting $R'$ of $R$ such that $\cx^{R'}(M,N)=\cx^R(M,N)-1$ and $\cx^{R'}M=\cx^RM-1$.  Thus by induction
\[
\cx^R(M,N)=\cx^{R'}(M,N)+1\le\cx^{R'}M+1=\cx^RM
\]

(2). For the first statement we induct on 
\[
r=\max\{\cx^R(M_1,N),\cx^R(M_2,N),\cx^R(M_3,N)\}
\]
If $r=0$, then the inequality holds.  Now assume $r>0$.  If $\cx^R(M_h,N)=0$, then the induced long exact sequence of Tor shows that $\cx^R(M_i,N)=\cx^R(M_j,N)$, and the inequality holds for all possible values of $h,i,j$.  Now assume all three complexities are positive.  Then by Theorem \ref{redcx},
one may find a principal lifting $R'$ of $R$ such that all three complexities are reduced by one.  By induction, the inequality holds over $R'$, thus also over $R$.  The second inequality is proved similarly.

(3). We induct on $r=\cx^RM$.  When $r=0$ the equality certainly holds.  Now assume $r>0$.  We claim that $\cx^R(M,M)>0$.  Otherwise, according to \cite[Theorem 3.3]{BeJo}, there exists a residual algebraic closure $Q^\sharp$ of $Q$ and a nonzerodivisor $f$ in the square of the maximal ideal of 
$Q^\sharp$ such that $\Tor^{Q^\sharp/(f)}_n(M\otimes_QQ^\sharp,M\otimes_QQ^\sharp)=0$ and $\Tor^{Q^\sharp/(f)}_n(M\otimes_QQ^\sharp,k\otimes_QQ^\sharp)\ne 0$ for all $n\gg 0$, and this contradicts Theorem \ref{AvrBuc}. Thus we know $\cx^R(M,M)>0$.  Therefore by Theorem \ref{redcx} we can find a principal lifting $R'$ of $R$ such that $\cx^{R'}(M,M)=\cx^R(M,M)-1$ and 
$\cx^{R'}M=\cx^RM-1$.  The result follows by induction.

(4). If $r=0$, then $M$ is perfect of grade $g$.  Thus, since $R$ is Gorenstein, 
we have that $\Ext^g_R(M,R)$ is also perfect of grade $g$.  In particular,
$\cx^R\Ext^g_R(M,R)=0$.  If $r>0$, then $\cx^R\Ext^g_R(M,R)>0$. Otherwise 
$M\cong\Ext^g_R(\Ext^g_R(M,R),R)$ would be perfect of grade $g$.  Thus by Theorem \ref{redcx}
we may find a principal lifting $R'$ of $R$ such that $\cx^{R'}M=\cx^RM-1$ and 
$\cx^{R'}\Ext^g_R(M,R)=\cx^R\Ext^g_R(M,R)-1$.  The result follows by induction.
\end{proof}

We also show that this invariant behaves well with respect to short exact sequences:
\begin{proposition} \label{biadd} The invariant $\theta_r^R(-,-)$ is biadditive on short exact sequences of finitely generated $R$-modules, when defined, whose pairs have complexity at most $r$.
\end{proposition}

\begin{proof} We induct on $r$.  When $r=0$, we have $\theta_0^R(-,-)=\chi^R(-,-)$, which is already known to be biadditive. When $r=1$ Tor modules are periodic.  Thus Hochster's proof in \cite[p. 98]{Hoc81} applies. Now assume that 
$r>1$. By Remark \ref{infk}, we may assume that $R$ has an infinite residue field.  Let $N$ be a finitely generated $R$-module and $0 \to M_1\to M_2 \to M_3 \to 0$ be a short exact sequence of finitely generated $R$-modules such that the complexities of the pairs $(M_i,N)$ are at most 
$r$, for $i=1,2,3$.  We now use Theorem \ref{redcx} and the fact that if the complexity of a pair of modules over $R$ is zero, then it is also zero over any principal lifting (see the long exact sequence \eqref{corles}): we may choose a principal lifting 
$R'$ of $R$ such that $\theta_{r-1}^{R'}(M_i,N)=\theta_r^R(M_i,N)$ for $i=1,2,3$, and all complexities over $R'$ are at most $r-1$.  Then by induction we have
\begin{align*}
&\theta_r^R(M_2,N)-\theta_r^R(M_1,N)-\theta_r^R(M_3,N) \\
&=\theta_{r-1}^{R'}(M_2,N)-\theta_{r-1}^{R'}(M_1,N)-\theta_{r-1}^{R'}(M_3,N)=0
\end{align*}
A similar argument applies in the other argument.
\end{proof}

\subsection*{Intersection inequalities}
We end this section addressing a famous open question originating from intersection theory of algebraic varieities (see \cite{PS}).

\begin{chunk}\label{QPSIT} Suppose that $Q$ is a local ring and $M$ and $N$ are finitely generated $Q$-modules with $\ell(M\otimes_QN)<\infty$ and $\pd_QM<\infty$.  Then does the following inequality hold?
\begin{equation}\label{PSIT}
\dim M+\dim N \le \dim Q
\end{equation}
\end{chunk}

As is done in \cite[Question 8.5]{Dao07}, one can ask whether an asymptotic version of \eqref{PSIT} holds:
\begin{chunk}\label{QPSAIT} Suppose that $M$ and $N$ are finitely generated $R$-modules with 
$\ell(M\otimes_RN)<\infty$ and $\Tor^Q_n(M,N)=0$ for all $n\gg 0$.  Then does the following inequality hold?
\begin{equation}\label{PSAIT}
\dim M+\dim N \le \dim R + \cx^RM
\end{equation}
\end{chunk}

\begin{proposition} Let $M$ and $N$ be finitely generated $R$-modules with $\ell(M\otimes_RN)<\infty$
and $\Tor^Q_n(M,N)=0$ for all $n\gg 0$.
Then \eqref{PSIT} holds if and only if \eqref{PSAIT} holds.
\end{proposition}

\begin{proof} We induct on $\cx^RM$.  When $\cx^RM=0$ both questions \ref{QPSIT} and 
\ref{QPSAIT} are the same question. Now assume that $\cx^RM>0$.  By Remark \ref{infk}, we may assume that $R$ has an infinite residue field. Then by Theorem \ref{redcx} there exists
a principal lifting $R'$ of $R$ such that $\cx^{R'}M=\cx^RM-1$.  Therefore induction gives that
$\dim M+\dim N \le \dim Q$ holds if and only if $\dim M+\dim N \le \dim R' + \cx^{R'}M$ holds, and the latter inequality is equivalent to $\dim M+\dim N \le \dim R + \cx^RM$.
\end{proof}

%%%%%%%%
\section{Depth instead of Dimension}\label{sec_depth}

If we use depth instead of dimension, we achieve a version of Theorem \ref{HochsterCI} where we can replace codimension by complexity and do not need to assume the vanishing of a theta invariant. In this section, $R=Q/(f_1,...,f_c)$ is assumed to be a complete intersection.

Recall from \cite{J1} the generalized Auslander-Buchsbaum formula: 
\begin{align*}
\sup\{i\mid & \Tor_i^R(M,N)\ne 0\}\\
&=\sup\{\depth R_p-\depth M_p-\depth N_p\mid p\in\supp M\cap\supp N\}
\end{align*}
which is shown to hold for finitely generated modules $M$ and $N$ over a complete intersection $R$, provided the left hand side is finite.  If $\ell(M\otimes_R N)<\infty$, then the maximal ideal of $R$ is the only prime in $\supp M\cap \supp N$, and so if $\cx^R(M,N)=0$, then the formula reduces to
\begin{equation}\label{GAB}
\sup\{i\mid \Tor_i^R(M,N)\ne 0\}=\depth R - \depth M - \depth N
\end{equation}
and one concludes that if $\ell(M\otimes_R N)<\infty$ and $\cx^R(M,N)=0$ then
\begin{equation}\label{depthineq}
\depth M + \depth N \le \depth R
\end{equation}
and if $R$ is a complete intersection, then it is Cohen-Macaulay, and $\depth R$ can be replaced by 
$\dim R$. Of course if $M$ and $N$ are also Cohen-Macaulay modules, then depth can replaced by dim in all the results below.

\begin{theorem}\label{depth_ineq} Assume that $R$ is a complete intersection, and that $M$ and $N$ are finitely generated $R$-modules with $\ell(M\otimes_R N)<\infty$. Then 
\[
\depth M+\depth N\le\dim R+\cx^R(M,N)\;.
\] 
\end{theorem}

\begin{proof} We induct on $r=\cx^R(M,N)$. When $r=0$, the statement is just \eqref{depthineq}. Now assuming $r>0$, by Remark \ref{infk}, we may assume that $R$ has an infinite residue field, and then by Theorem \ref{redcx} there exists a principal lifting $R'$ of $R$ such that 
$\cx^{R'}(M,N)=r-1$.  Thus by induction we have 
\[
\depth M + \depth N \leq \depth R' + r-1=\depth R+r
\]
as required.
\end{proof}

Consider the invariant $\q^R(M,N)=\sup\{i\mid \Tor_i^R(M,N)\ne 0\}$.

\begin{chunk}\label{qlift} By \cite[Lemma 2.9(2)]{J1}, when $\q^R(M,N)$ is finite we have
\[
\q^Q(M,N)=\q^R(M,N)+c
\]
\end{chunk}

Using the invariant $\q^Q(M,N)$, we have a more precise version of Theorem \ref{depth_ineq}.

\begin{theorem}\label{theform} Assume that $R$ is a complete intersection, and that $M$ and $N$ are finitely generated $R$-modules with $\ell(M\otimes_R N)<\infty$. Then 
\[
\q^Q(M,N)-c=\dim R-\depth M-\depth N
\] 
\end{theorem}

\begin{proof} We have from \eqref{GAB} and \ref{qlift}
\begin{align}
\q^Q(M,N)&=\dim Q - \depth M - \depth N\\
&=(\dim R +c)- \depth M - \depth N
\end{align}
which gives the result.
\end{proof}

Thus we get the following.

\begin{corollary} Assume that $R$ is a complete intersection, and that $M$ and $N$ are finitely generated $R$-modules with 
$\ell(M\otimes_R N)<\infty$. Then 
\[
\depth M+\depth N\le\dim R
\] 
if and only if $\q^Q(M,N)\ge c$.
\end{corollary}

From the definition of $\q^Q(M,N)$ and Theorems \ref{depth_ineq} and \ref{theform}, there are bounds on this invariant:

\begin{chunk} For a complete intersection $R$ and finitely generated $R$-modules $M$ and $N$ satisfying $\ell(M\otimes_R N)<\infty$, we have
\[
c-\cx^R(M,N)\le \q^Q(M,N)\le \dim Q \;.
\]
\end{chunk}

There are examples showing that these bounds are sharp for pairs of modules $M$ and $N$ satisfying $\ell(M\otimes_R N)<\infty$. Indeed, taking both $M=k$ and $N=k$ gives $\q^Q(M,N)=\dim Q$. For the other extreme we employ a well-used example from \cite[Example 4.1]{J2}.

\begin{example} Let $k$ be a field, $Q=k[x_1,\dots,x_n,y_1,\dots,y_n]$ and 
\[
R=Q/(x_1y_1,\dots,x_ny_n)
\] 
Then for the $R$-modules $M=R/(x_1,\dots,x_n)$ and $N=R/(y_1,\dots,y_n)$ we have $\ell(M\otimes_RN)<\infty$, $\cx^R(M,N)=n=c$ and $\q^Q(M,N)=0$.
\end{example}

The following is an extension of \cite[Corollary V.B.6]{Serre}.

\begin{corollary} Assume that $R$ is a complete intersection, and that $M$ and $N$ are finitely generated $R$-modules with $\ell(M\otimes_RN)<\infty$. Then 
\[
\q^Q(M,N)\ge \dim Q-\dim M - \dim N
\]
and equality holds if and only if $M$ and $N$ are Cohen-Macaulay.
\end{corollary}

\begin{proof}
From \ref{GAB} (with $R=Q$) we have
\begin{align*}
\q^Q(M,N)&=\dim Q-\depth M-\depth N\\
&=\dim Q-\dim M - \dim N + (\dim M-\depth M) + (\dim N-\depth N)
\end{align*}
Since the quantities in parentheses are nonnegative, the inequality follows; the quantities are zero if and only if equality holds.
\end{proof}

We consider the Intersection Theorem of Peskine-Szpiro and Roberts: Suppose that $A$ is a local ring, and $M$ and $N$ are finitely generated $A$-modules such that $\pd_AM<\infty$ and $\ell(M\otimes_AN)<\infty$.  Then $\dim N \le \pd_AM$, or, using the Auslander-Buchsbaum Formula, this can be rewritten as
\begin{equation}\label{IT}
\dim N + \depth M \le \depth A 
\end{equation}

We may extend this as follows.

\begin{theorem}\label{int_thm_ext} Assume that $R$ is a complete intersection, and that $M$ and $N$ are finitely generated $R$-modules with $\ell(M\otimes_RN)<\infty$. Then
\begin{equation}
\dim N +\depth M\le \dim R +\cx^RM \;.
\end{equation}
\end{theorem}

\begin{proof} We induct on $s=\cx^RM$. The case $s=0$ is equivalent to $\pd_RM<\infty$.  Thus in this case \eqref{IT} gives the desired inequality.

Now assume $s>0$.  By Remark \ref{infk}, we may assume that $R$ has an infinite residue field. Then by Theorem \ref{redcx} there exists a principal lifting $R'$ of $R$ such that $\cx^{R'}M=s-1$.  Thus by induction we have  
\[
\dim N +\depth M\le \dim R' +\cx^{R'}M=(\dim R+1)+(s-1)=\dim R+s
\]
as desired.
\end{proof}

%%%%%%%%
\section{Depth versions of Serre's conjectures in higher codimension}\label{sec_highercodim} 

In this section, assume that $R=Q/(f_1,...,f_c)$ is a complete intersection. Let $M$ and $N$ be finitely generated $R$-modules with $\ell(M\otimes_R N)<\infty$ and $r=\cx^R(M,N)$. In Section \ref{sec_theta}, we suggest that $\theta_c^R$ should play the role of $\chi^Q$ in a higher codimension analogue of Serre's intersection multiplicity conjectures. On the other hand, one observation of Section \ref{sec_depth} is that shifting focus from dimension to depth, and from codimension to complexity, provides another higher codimension analogue of Serre's dimension inequality (see Theorem \ref{depth_ineq}): 
\begin{equation}\label{eqn_depthineq}\depth M + \depth N \leq \dim R + r\end{equation}
Thus we are naturally led to ask: Does equality in \eqref{eqn_depthineq} imply positivity of $\theta_r^R(M,N)$? We confirm this here, and remark that the converse does not hold. We also modify an example of Dutta, Hochster, and McLaughlin to show that $\theta_r^R(M,N)$ be can be negative.

We start with a higher codimension ``depth" analogue of the positivity conjecture: 

\begin{theorem} \label{depth_ineq_strict_vanishing}
Assume that $R$ is a complete intersection, and that $M$ and $N$ are finitely generated $R$-modules with $\ell(M\otimes_R N)<\infty$. Set $r=\cx^R(M,N)$.  Consider the following conditions
 \begin{enumerate}
 \item $\theta_r^R(M,N)\leq 0$;
 \item $\depth M+\depth N<\dim R+r$. 
 \end{enumerate}
Then $(1)$ implies $(2)$. 
\end{theorem}
\begin{proof}
We will show that if $\depth M+\depth N=\dim R+r$, then $\theta_r^R(M,N)>0$. 
First consider the case where $r=0$. In this case, one has that the assumption implies by \eqref{GAB} that $\Tor_i^R(M,N)=0$ for $i>0$. Thus 
$$\theta_0^R(M,N)=\chi^R(M,N)=\ell(M\otimes_RN)>0.$$ 
Now consider the case when $r>0$. By Remark \ref{infk}, we may assume that $R$ has an infinite residue field. By Theorem \ref{redcx}, there exists a principal lifting $R'$ of $R$ such that 
$$\cx^{R'}(M,N)=\cx^R(M,N)-1\quad\text{and}\quad\theta_{r-1}^{R'}(M,N)=\theta_r^R(M,N).$$ 
Thus $\dim R+r=\dim R'+\cx^{R'}(M,N)$. 
Inductively we may thus reduce to the case where $\cx^R(M,N)=0$, and we are done by the first part.
\end{proof}

\begin{remark}\label{converse}
The implication $(2)\Rightarrow(1)$ in Theorem \ref{depth_ineq_strict_vanishing} does not always hold: over a regular local ring $Q$, take any pair of finitely generated $Q$-modules $M$ and $N$, at least one of which is not Cohen-Macaulay, having $\chi^Q(M,N)>0$. 
(For example, one could take $Q=k[[x,y]]$, $M=Q/(x^2,xy)$, and $N=Q/(y)$, and note that 
$\chi^Q(M,N)=1$.) In this situation, one would have $\depth M+\depth N<\dim Q$, yet $\theta_0^Q(M,N)>0$. Thus the corresponding analogue of the vanishing conjecture for $\theta_r^R$ relative to the inequality \eqref{eqn_depthineq} fails.
\end{remark} 

It is still natural to wonder whether an analogue of the vanishing conjecture for $\theta_r^R$ relative to the inequality \eqref{eqn_depthineq} holds for Cohen-Macaulay modules, but this is also not the case. When $r=0$, the example from Dutta-Hochster-McLaughlin \cite{DHM85} provides a hypersurface $R$ and $R$-modules $M$ and $N$, where $\depth M=0$, $\depth N=2$, $\dim R=3$, and $r=0$, but $\theta_0^R(M,N)=\chi^R(M,N)<0$. One can modify that same example to obtain a similar counterexample for $\theta_1^R$:

\begin{example}\label{DHMex}
Let $R'$ be the local hypersurface from \cite{DHM85}: 
$$R'=k[X_1,X_2,X_3,X_4]_{(X_1,X_2,X_3,X_4)}/(X_1X_4-X_2X_3)$$ 
and consider the local complete intersection $R=R'/(x_2^2)$ having codimension 2 and dimension 2. Let $M$ be the $R'$-module of finite projective dimension and finite length from \cite{DHM85}; it is also an $R$-module as it is killed by $x_2^2$. Let $N=R/(x_1,x_2)$ and note $\depth N=2=\dim N$.

It is clear that $r=\cx^R(M,N)$ is at most 1. To see that it cannot be 0, note that $R'$ is a principal lifting of 
$R$ such that $\cx^{R'}(M,N)=0$. Now by Lemma \ref{chi=theta} and the fact that $\chi^{R'}(M,N)=-1$, we see that $\cx^{R}(M,N)\ne 0$. Again by Lemma \ref{chi=theta} we have 
$\theta_1^{R}(M,N)=\chi^{R'}(M,N)=-1$. Thus the implication 
\[
\dim M + \dim N < \dim R + r \quad \implies \quad \theta^R_1(M,N)=0
\]
fails for a codimension 2 complete intersection $R$.
\end{example}
\begin{remark}\label{2to1fails}
This example also shows that the implication $(2)\Rightarrow (1)$ in Theorem \ref{HochsterCI} does not hold if one only replaces the codimension $c$ by the complexity $r$ in both conditions of that theorem, even for pairs of Cohen-Macaulay modules.
\end{remark}
The following still seems reasonable to ask, however:
\begin{question}
Assume that $\ell(M\otimes_R N)<\infty$.
If $\dim M+\dim N<\dim R+r$, does one always have $\theta_r^R(M,N)\leq 0$?
\end{question}

%%%%%%%%%%%
\section{Lifting pairs of modules}\label{sec_lift}
In this section we introduce a condition on a pair of $R$-modules under which a version of Theorem \ref{HochsterCI} involving the complexity of the pair of modules holds; cf. Remark \ref{2to1fails}. Let $R=Q/(f_1,...,f_c)$ be a complete intersection as in Notation \ref{Rnotation}.

\begin{definition}\label{int_lift_dfn}
Let $M,N$ be a pair of $R$-modules with $\cx^R(M,N)=0$ and such that $\ell(M\otimes_R N)<\infty$. An \emph{intersection lifting of $M,N$ to $Q$} is a pair of $Q$-modules $M',N'$ such that $\ell(M'\otimes_Q N')<\infty$ and satisfying:
\begin{enumerate}
\item $\dim M'+\dim N'=\dim M+\dim N+c$, and
\item $\chi^Q(M',N')=\chi^R(M,N)$.
\end{enumerate}
\end{definition}

\begin{definition}
A ring $R'$ is an \emph{intermediate complete intersection} of $R$ of codimension $c'$ if there exists a set of generators $g_1,...,g_c$ of $(f_1,...,f_c)$ such that $R'=Q/(g_1,...,g_{c'})$. In this case, we have $R=R'/(g_{c'+1},...,g_c)$. 
\end{definition}

\begin{definition}\label{dfn_int_lift}
Let $M,N$ be a pair of $R$-modules. The pair $M,N$ is \emph{intersection liftable to $Q$} if for any intermediate complete intersection $R'$ with $\cx^{R'}(M,N)=0$, there exists an intersection lifting to $Q$ of the pair of $R'$-modules $M,N$.
\end{definition}

One may compare the next result to Serre's intersection multiplicity conjectures recalled in \ref{SerreIMC}. 

\begin{theorem}\label{thm_liftable}
Assume that $R=Q/(f_1,...,f_c)$ is a complete intersection. Let $M,N$ be a pair of $R$-modules with $\cx^R(M,N)=r$. If $M,N$ is intersection liftable to $Q$, then the following hold:
\begin{enumerate}
\item[$(1)$] $\dim M+\dim N\le\dim R+r$,
\item[$(2)$] $\theta_r^R(M,N)\geq 0$,
\item[$(3)$] If $\dim M+\dim N<\dim R+r$, then $\theta_r^R(M,N)=0$,
\item[$(4)$] If $\dim M+\dim N=\dim R+r$ and $Q$ satisfies Serre's positivity conjecture, then $\theta_r^R(M,N)>0$.
\end{enumerate}
\end{theorem}

\begin{proof}
By Remark \ref{infk}, we may assume that $R$ has an infinite residue field. Repeatedly applying Theorem \ref{redcx}, we obtain an intermediate complete intersection $R'$ of codimension $c-r$ such that $\cx^{R'}(M,N)=0$ and $\theta_r^R(M,N)=\theta_0^{R'}(M,N)$. By intersection liftability, one has that there exist $Q$-modules $M'$ and $N'$ with $\ell(M'\otimes_Q N')<\infty$ and such that 
$$\dim M'+\dim N'=\dim M+\dim N+c-r\quad\text{and}\quad\chi^{Q}(M',N')=\chi^{R'}(M,N)\;.$$ 

The proofs of (1)-(4) all reduce to the corresponding results in \ref{SerreIMC}. First notice that $\dim M+\dim N\leq \dim R+r$ is equivalent to $\dim M'+ \dim N'\leq \dim Q$, so (1) holds. One now has $\theta_r^R(M,N)=\chi^{R'}(M,N)$, thus $\theta_r^R(M,N)=\chi^Q(M',N')$. Non-negativity of $\chi^Q$ implies the same for $\theta_r^R$, so that (2) holds.  Since $\dim M+\dim N<\dim R+r$ is equivalent to $\dim M'+\dim N'<\dim Q$, (3) and (4) follow similarly.
\end{proof}

A first example of intersection liftability is when one of the modules is cyclic and defined by a regular sequence.

\begin{example}
Let $M,N$ be a pair of $R$-modules with $\cx^R(M,N)=0$ and $\ell(M\otimes_R N)<\infty$, and suppose that 
$M=R/(g_1,...,g_n)$, where $g_1,...,g_n$ is an $R$-regular sequence. As $Q$ is local, the sequence $g_1,...,g_n$ lifts to a regular sequence in $Q$, which (by abuse of notation) we also call $g_1,...,g_n$. Set $M'=Q/(g_1,...,g_n)$ and $N'=N$. 

We claim that $M',N'$ is an intersection lifting of $M,N$ to $Q$. As $f_1,...,f_c$ is $M'$-regular, and $M=M'/(f_1,...,f_c)M'$, evidently $\dim M'=\dim M+c$, so condition (1) of Definition 
\ref{int_lift_dfn} holds. The Koszul complex $\Kos^Q(g_1,...,g_n)$ is a free resolution of $M'$ over 
$Q$. In fact, $\Kos^Q(g_1,...,g_n)\otimes_QR\cong \Kos^R(g_1,...,g_n)$ since $g_1,...,g_n$ is $R$-regular. Thus for $i\geq 0$,
\begin{align*}
\Tor_i^Q(M',N')&=\HH_i(\Kos^Q(g_1,...,g_n)\otimes_Q N')\\
&\cong \HH_i(\Kos^R(g_1,...,g_n)\otimes_R N)\\
&\cong \Tor_i^R(M,N)\;.
\end{align*}
The second condition of Definition \ref{int_lift_dfn}, that $\chi^Q(M',N')=\chi^R(M,N)$, therefore follows.
\end{example}

In the following discussion we give a common way in which a pair of $R$-modules $M,N$ has an intersection lifting. The approach generalizes the previous example.

\begin{definition} We say that the $R$-module $M$ \emph{lifts} to $Q$ if there exists a $Q$-module $M'$
such that $f_1,\dots,f_c$ is a regular sequence on $M'$ and $M'\otimes_QR\cong M$. In this case we call
$M'$ a \emph{lifting of $M$ to $Q$}.
\end{definition}

Note that if $M'$ is a lifting of $M$ to $Q$, then 
\[
\depth M'=\depth M+c
\]

We recall an important fact regarding change of rings for Tor, \cite[11.51]{R}).

\begin{chunk}\label{stand} Suppose that $A$ is a commutative ring, $J$ an ideal
of $A$, and set $B=A/J$. If $X$ is an $A$-module such that
$\Tor^A_i(X,B)=0$ for all $i\geq 1$, then for any $B$-module $Y$ we
have
\[
\Tor^A_i(X,Y)\cong\Tor^B_i(X\otimes_A B,Y)\quad\text{for all $i\ge 0$}\;.
\]
\end{chunk}

\begin{theorem}\label{semilifting}  
Let $M$ and $N$ be finitely generated $R$-modules. Suppose that for some $c'$, $0\le c'\le c$, $M$ lifts to $R_1=Q/(f_1,\dots,f_{c'})$ with lifting $M'$ and 
$N$ lifts to $R_2=Q/(f_{c'+1},\dots,f_c)$ with lifting $N'$. Then
\begin{enumerate}
\item $\dim M'+\dim N'=\dim M+\dim N+c$, and
\item $\Tor^Q_i(M',N')\cong\Tor^R_i(M,N)$ for all $i\ge 0$.
\end{enumerate}
In particular, $\cx^R(M,N)=0$ if $\cx^Q(M',N')=0$.
\end{theorem}

\begin{proof} Since $M'$ is a lifting of $M$ to $R_1$, we have $\depth M'=\depth M+c-c'$, and since
$N'$ is a lifting of $N$ to $R_2$, we have $\depth N'=\depth N+c'$. Thus (1) follows.

For (2), we note that $\Tor^{R_1}_i(M',R)$ is the Koszul homology of the Koszul complex
$\Kos^{R_1}(\overline f_{c'+1},\dots,\overline f_c)\otimes_{R_1}M'$. Since $\overline f_{c'+1},\dots,\overline f_c$ is regular on $M'$, we have $\Tor^{R_1}_i(M',R)=0$ for all $i\ge 1$. Thus \ref{stand} says that 
\begin{equation}\label{T1}
\Tor^{R_1}_i(M',N)\cong\Tor^R_i(M,N) \text{ for all $i\ge 0$}
\end{equation}
Since $\overline f_1,\dots,\overline f_{c'}$ in $R_2$ is regular on $N'$, it is also true that
$f_1,\dots,f_{c'}$ in $Q$ is regular on $N'$.  We have that $\Tor^Q_i(R_1,N')$ is the Koszul homology of the Koszul complex $\Kos^Q(f_1,\dots,f_{c'})\otimes_Q N'$.  Thus $\Tor^Q_i(R_1,N')=0$ for all 
$i\ge 1$. Now \ref{stand} gives
\begin{equation}\label{T2}
\Tor^Q_i(M',N')\cong \Tor^{R_1}_i(M',N) \text{ for all $i\ge 0$}
\end{equation}
Thus \eqref{T1} and \eqref{T2} give (2).
\end{proof}

\begin{corollary} Let $M$ and $N$ be finitely generated $R$-modules with $\ell(M\otimes_R N)<\infty$. Suppose that for some $c'$, $0\le c'\le c$, $M$ lifts to $R_1=Q/(f_1,\dots,f_{c'})$ with lifting $M'$ and $N$ lifts to $R_2=Q/(f_{c'+1},\dots,f_c)$ with lifting $N'$. Then the pair 
$M',N'$ is an intersection lifting of $M,N$.
\end{corollary}

Here is an example illustrating Theorem \ref{semilifting}.  We use the same ring as in Example 
\ref{DHMex}.

\begin{example}
Consider the ring $Q=k[X_1,X_2,X_3,X_4]_{(X_1,X_2,X_3,X_4)}$, and set $R=Q/(X_1X_4-X_2X_3,X_2^2)$,
$M=R/(x_1,x_3,x_4)$, and $N=R/(x_2)$.  Then clearly $M\otimes_RN$ has finite length. One checks easily that $M$ lifts to $R_1=Q/(X_1X_4-X_2X_3)$, with lifting $M'=Q/(X_1,X_3,X_4)$, and $N$ lifts to
$R_2=Q/(X_2^2)$, with lifting $N'=Q/(X_2)$.  Thus Theorem \ref{semilifting} applies.
\end{example}

%%%%%%%%
\section{Appendix}

Below are important results that were deferred from the main body of the paper.

%\subsection{Principal liftings}\label{pls}
\begin{chunk}
{\bf Principal liftings.}\label{pls}
Let $(R',\n',k)$ be a local ring, $f$ a nonzerodivisor in $\n'^2$, and denote by $R$ the quotient 
$R'/(f)$. The following construction from \cite[Theorem 3.2]{BeJo} yields an $R'$-free resolution of an $R$-module $M$ from an $R$-free resolution. 

Let $F = (F_n, \partial_n)$ be an $R$-free resolution of an $R$-module $M$.
We lift $F$ to a sequence $\widetilde{F} = ( \widetilde{F}_n, \widetilde{\partial}_n )$ of free $R'$-modules and maps. For each $n \ge 0$ let 
$\widetilde{t}_{n+2} \colon \widetilde{F}_{n+2}  \to  \widetilde{F}_{n}$ be the $R'$-homomorphism with 
$\widetilde{\partial}_{n+1} \circ \widetilde{\partial}_{n+2} = f \cdot \widetilde{t}_{n+2}$. Then the sequence $F'$ given by
$$\xymatrix@C=50pt{
\cdots \ar[r] & \widetilde{F}_3 \oplus \widetilde{F}_2 \ar[r]^{\left [ \begin{smallmatrix} \widetilde{\partial}_3 & -f \\ \widetilde{t}_3 & -\widetilde{\partial}_2 \end{smallmatrix} \right ]} & \widetilde{F}_2 \oplus \widetilde{F}_1 \ar[r]^{\left [ \begin{smallmatrix} \widetilde{\partial}_2 & -f \\ \widetilde{t}_2 & -\widetilde{\partial}_1 \end{smallmatrix} \right ]} & \widetilde{F}_1 \oplus \widetilde{F}_0 \ar[r]^{\left [ \begin{smallmatrix} \widetilde{\partial}_1 & -f \end{smallmatrix} \right ]} & \widetilde{F}_0 }$$
is an $R'$-free resolution of $M$.

As is proved in \cite{Eis80}, the family of maps $t=\{t_n=\widetilde t_n\otimes R:F_n\to F_{n-2}\}$ defines a degree $-2$ chain endomorphism on $F$. For degree considerations, we write 
$t:\Sigma^{-1}F\to \Sigma F$, where $\Sigma^sF$ denotes the shifted complex with 
$(\Sigma^sF)_i=F_{i-s}$. Let $C$ be the cone of $t$.  Then we have the standard short exact sequence of complexes $0\to \Sigma F \to C\to F\to 0$.  It is easy to see that $C$ is isomorphic to 
$F'\otimes_{R'}R$ as complexes.
Thus, for an $R$-module $N$, the \emph{change of rings long exact sequence} of homology of
\[
0\to \Sigma F\otimes_R N \to C\otimes_R N \to F\otimes_RN \to 0
\]
is
\begin{align}
&\hskip 1.2in \vdots \label{corles}\\
&\Tor^R_2(M,N)\to \Tor^{R'}_3(M,N)\to\Tor^R_3(M,N)\xra x \nonumber\\
&\Tor^R_1(M,N)\to \Tor^{R'}_2(M,N)\to\Tor^R_2(M,N)\xra x \nonumber\\
&\Tor^R_0(M,N) \to\Tor^{R'}_1(M,N)\to \Tor^R_1(M,N)\to 0\nonumber
\end{align}
where the maps $x$ are the $\Tor^R_i(t,N)$.
\end{chunk}

%\subsection{Eisenbud operators and the polynomial ring $\mathcal R=R[x_1,\dots,x_c]$}\label{eisops}
\begin{chunk}\label{eisops}
{\bf Eisenbud operators and the polynomial ring $\mathcal R=R[x_1,\dots,x_c]$.}
In this subsection, Notation \ref{Rnotation} is in force. We discuss the Eisenbud operators as per \cite{Eis80}. Let $F=(F_n,\partial_n)$ be a free resolution of the $R$-module $M$.  Similar to what we did in \ref{pls}, we lift this resolution to $Q$ as a sequence of homomorphisms of 
$Q$-free modules $\widetilde F=(\widetilde F_n,\widetilde\partial_n)$.  Since $\widetilde F$ is a complex modulo $(f_1,...,f_c)$, we can write $\widetilde\partial^2=\sum f_i\widetilde t_i$, where the $\widetilde t_i$ are endomorphisms of $\widetilde F$ of degree $-2$.  Then as is discussed in \cite{Eis80}, the endomorphisms $t_i=\widetilde t_i\otimes R$ of $F$ are of degree $-2$, are well defined up to homotopy, and commute with each other, up to homotopy.  Finally, the action of $x_i$ on 
$\Tor^R(M,N)$ is given by $x_i=\Tor^R(t_i,N)$.

Now let $x$ be a linear form of $\mathcal R_{-2}$ that is eventually surjective on $\Tor^R(M,N)$.  Write 
$x=a_1x_1+\cdots +a_cx_c$ for $a_i\in R$.  Then, appealing to Nakayama's lemma, at least one of the $a_i$ must be a unit in $R$.  Without loss of generality, assume $a_c$ is a unit in 
$R$. Let $\widetilde a_i$ denote a preimage in $Q$ of $a_i$, $1\le i\le c$.  Note that 
$\widetilde a_c$ is also a unit in $Q$. Set $g_i=f_i-(\widetilde a_i/\widetilde a_c)f_c$ for
$1\le i\le c-1$ and $g_c=(1/\widetilde a_c)f_c$.  Then $g_1,\dots,g_c$ is another minimal generating set for $(f_1,...,f_c)$.  Letting $R'=Q/(g_1,\dots,g_{c-1})$ we see that $t$ from \ref{pls} is nothing more than
$a_1t_1+\cdots +a_ct_c$.  Thus the $x$ from \ref{pls} is eventually surjective.
\end{chunk}

%\subsection{Proof of Theorem \ref{Kirby}}
\begin{chunk}
{\bf Proof of Theorem \ref{Kirby}.}
\begin{proof}  Kirby gives a proof in the case where all the $x_i$ have degree $1$.  The obvious modifications are made when assuming all the $x_i$ have degree $-1$; this is the statement of Theorem \ref{Kirby} assuming $d=-1$.  For the general case where all the $x_i$ have the same negative degree
$d<0$, the graded module $G$ naturally decomposes as a direct sum of $d$ graded submodules
$G=G^{(0)}\oplus\cdots\oplus G^{(d-1)}$, where $G^{(j)}=\bigoplus_{i\in\mathbb Z}G_{di+j}$. After reindexing, we can assume each $G^{(j)}$ is a module over a polynomial ring in variables of degree 
$-1$, thus obtaining integers $r^{(j)}$ and $s^{(j)}$ satisfying the conditions of the theorem for 
$d=-1$. By taking the appropriate extreme values of these integers, and indexing back, we obtain the integers $r$ and $s$ satisfying the conditions of Theorem \ref{Kirby} in the general case.
\end{proof}
\end{chunk}

%\subsection{Existence of eventually regular elements}
\begin{chunk}
{\bf Existence of eventually regular elements.}
The following definition and theorem are typically stated for nonnegatively graded rings, but for our purposes we need versions for nonpositively graded rings. 

Recall that $\R=R[x_1,...,x_c]$ is the polynomial ring in $c$ variables, each of degree $-2$, as in Assumptions \ref{Gu}.  
Let $G$ be a graded $\R$-module.  We say that $x\in \R$ is \emph{eventually regular} on $G$
if $(0:_Gx)_n=0$ for all $n\ll 0$.

The following is a nonpositively graded variation of \cite[Page 285]{Zariski-Samuel}; therein eventually regular elements are referred to as \emph{superficial}.  Another synonym of eventually regular occurring in the literature is \emph{filter regular}, see \cite{BIKO}.

\begin{theorem}\label{eres} 
Assume that $(R,\fm,k)$ has an infinite residue field $k$. Let $G$ be a nonpositively graded finitely generated $\R$-module. Then there exists $x\in \R_{-2}$ which is eventually regular on $G$.
\end{theorem}

\begin{proof}  Let 
\[
0=Q_1\cap \cdots\cap Q_s
\]
be a (graded) primary decomposition of the zero submodule of $G$, and let
$p_i=\sqrt{\Ann_{\R}(G/Q_i)}$ be the prime ideal associated to $Q_i$.  

First assume that ${\R}_{-2}\subset p_i$ for all $1\le i\le s$. It follows that there exists $a\ge 1$ such that 
\[
({\R}_-)^a\subseteq \Ann_{\R}(G/Q_1)\cap \cdots \cap \Ann_{\R}(G/Q_s)
\]
where ${\R}_-$ denotes $\R_{\leq -1}$.  In other words,
\[
({\R}_-)^aG\subseteq Q_1\cap \cdots \cap Q_s=0\;.
\]
This can only happen if $G_n=0$ for all $n\ll 0$, and so any $x\in {\R}_{-2}$ will do the job.

Therefore we may assume that for some $1\le h\le s$, 
${\R}_{-2}\not\subset p_i$ for $1\le i\le h$ and 
${\R}_{-2} \subset p_i$ for $h+1\le i\le s$.  Let $S_i=p_i\cap {\R}_{-2}$
for $1\le i\le h$.  Then, employing Nakayama's lemma, $(S_i+\fm {\R}_{-2})/\fm {\R}_{-2}$ is a proper
subspace of ${\R}_{-2}/\fm {\R}_{-2}$ for all $1\le i \le h$. As $k$ is infinite, the vector space $\R_{-2}/\fm \R_{-2}$ cannot be a union of finitely many proper subspaces. Thus there
exists $x\in {\R}_{-2}$ such that $x+\fm {\R}_{-2}$ is not in 
$(S_i+\fm {\R}_{-2})/\fm {\R}_{-2}$ for all $1\le i\le h$.  In other words,
$x\in {\R}_{-2}\setminus p_i$ for all $1\le i\le h$.

Similar to the argument above, there exists an integer $m<0$ such that 
$G_n\subseteq Q_{h+1}\cap\cdots\cap Q_s$ for all $n\le m$. Let 
$u\in (0:_Gx)_n$ for some $n\le m$.  Then $u\in Q_{h+1}\cap\cdots\cap Q_s$.
Since $xu=0\in Q_1\cap \cdots\cap Q_h$, and $x$ is not in $p_i$ for 
$1\le i\le h$, it follows from the definition of primary submodules that $u\in Q_1\cap \cdots\cap Q_h$. Thus
$u\in Q_1\cap\cdots \cap Q_s=0$.
\end{proof}

\end{chunk}

\section*{Acknowledgements} 
We are thankful for the anonymous referee's careful reading and detailed comments that greatly helped to improve the exposition of the paper.

\end{document}